\documentclass{article}
\usepackage[cp1251]{inputenc}
\usepackage[russian]{babel}
\usepackage{amssymb,amsfonts,amsmath,amsthm,amscd}
\usepackage{graphicx}
\usepackage{enumerate}
\usepackage{soul}
\usepackage[matrix,arrow,curve]{xy}
\usepackage{longtable}
\usepackage{array}
\usepackage{bm}

\newdimen\symskip
\newdimen\defskip
\newdimen\parind
\parind=\parindent
\newdimen\leftmarge
\newdimen\theoremshape
\topsep\defskip
\makeatletter

\newcommand*{\clei}{\nobreak\hskip\z@skip}

\renewcommand{\:}{\textup{:}}
\renewcommand{\~}{\textup{;}}
\DeclareRobustCommand*{\ti}{~\textemdash{} }
\DeclareRobustCommand*{\dh}{\clei\hbox{-}\clei}
\newcommand{\No}{№\,}
\def\thempfn{\ifcase\value{footnote}1\or *\or **\or ***\else\@ctrerr\fi}
\renewcommand\footnoterule{%
  \kern-3\p@
  \hrule\@width1in
  \kern2.6\p@}
\@addtoreset{footnote}{section}
\@addtoreset{footnote}{page}

\renewcommand{\@biblabel}[1]{[#1]}
\renewenvironment{thebibliography}[1]
     {\renewcommand{\refname}{References}%
      \section*{\refname}%
      \@mkboth{\MakeUppercase\refname}{\MakeUppercase\refname}%
      \list{\@biblabel{\@arabic\c@enumiv}}%
           {\itemsep\baselineskip
            \leftmargin\parind
            \settowidth\labelwidth{\@biblabel{#1}}%
            \labelsep\parind\advance\labelsep-\labelwidth
            \@openbib@code
            \usecounter{enumiv}%
            \let\p@enumiv\@empty
            \renewcommand\theenumiv{\@arabic\c@enumiv}}%
      \sloppy
      \clubpenalty4000
      \@clubpenalty\clubpenalty
      \widowpenalty4000%
      \sfcode`\.\@m}
     {\def\@noitemerr
       {\@latex@warning{Empty `thebibliography' environment}}%
      \endlist}

\def\@maketitle{%
  \newpage
  \vskip0.5em%
  UDK \udk%
  \vskip0.5em%
  MSC \msc%
  \vskip1.5em%
  \begin{center}\bf%
  \let\footnote\thanks%
   {\Large\@author\par}%
   \vskip1em%
   {\LARGE\@title\par}%
   \vskip1em%
   {\large\@date}%
  \end{center}%
  \par
  \vskip1.5em}

\def\@title{\@latex@warning@no@line{No \noexpand\title given}}

\renewcommand\sectionmark[1]{%
 \markright{%
  \ifnum \c@secnumdepth >\z@
   \thesection. \ %
  \fi
 #1}}%
\renewcommand{\section}{\@startsection{section}{1}{0pt}%
{5.5ex plus .5ex minus .2ex}{1.5ex plus .3ex}%
{\center\normalfont\Large\bfseries\sffamily\boldmath}}
\newcommand{\Ss}{\textup{\S\,}}

\def\@postskip@{\hskip.5em\relax}

\def\postsection{.\@postskip@}
\def\@seccntformat#1{\csname pre#1\endcsname\csname the#1\endcsname\csname post#1\endcsname}
\renewcommand{\thesection}{\textup{\arabic{section}}}

\newcommand{\parr}{\par\addvspace{\defskip}}
\newcommand{\theo}[2]{\newtheorem{#1}{#2}}
\newcommand{\deff}[2]{\newenvironment{#1}{\parr\textbf{#2.}}{\parr}}
\theo{theorem}{Theorem}
\theo{lemma}{Lemma}
\theo{prop}{Proposition}
\theo{stm}{Statement}
\theo{imp}{Corollary}
\deff{df}{Definition}
\deff{note}{Note}

\def\@begintheorem#1#2[#3]{%
  \deferred@thm@head{\the\thm@headfont \thm@indent
    \@ifempty{#1}{\let\thmname\@gobble}{\let\thmname\@iden}%
    \@ifempty{#2}{\let\thmnumber\@gobble}{\let\thmnumber\@iden}%
    \@ifempty{#3}{\let\thmnote\@gobble}{\let\thmnote\@iden}%
    \thm@notefont{\bfseries\upshape}%
    \indent%
    \thm@swap\swappedhead\thmhead{#1}{#2}{#3}%
    \the\thm@headpunct
    \thmheadnl 
    \hskip\thm@headsep
  }%
  \ignorespaces}
\renewenvironment{proof}{\parr\pushQED{\qed}\normalfont$\square\quad$}{\popQED\@endpefalse\parr}

\newcounter{col}
\newcounter{coll}
\newcommand{\mt}[3]{\multicolumn{#1}{#2}{#3}}
\newcommand{\news}{\\\hline}
\newcommand{\refcol}[1]{\addtocounter{col}{#1}}
\newcommand{\fcol}{\setcounter{coll}{\value{col}}}
\newcommand{\mc}[3]{&\mco{#1}{#2}{#3}}
\newcommand{\nc}[2]{\mc{1}{#1}{#2}}
\newcommand{\mco}[3]{\mt{#1}{#2|@{\refcol{#1}}}{#3}}
\newcommand{\mcu}[3]{\mco{#1}{|#2}{#3}}
\newcommand{\ncu}[2]{\mcu{1}{#1}{#2}}
\newcommand{\mtco}[2]{\mt{\value{coll}}{#1}{#2}}
\newcommand{\mtc}[1]{\mtco{c}{#1}}
\newcommand{\mtcc}[1]{\mtco{|c|}{#1}}
\newcommand{\nwl}{\newline}
\newcommand{\lonu}[3]{%
\setcounter{col}{0}
\begin{longtable}{#1}
\hline#2
\fcol
\endfirsthead
\hline#2
\endhead
\mtc{}\\
\mtc{\textit{Continuation on the next page}}
\endfoot
\endlastfoot
\hline#3\news\end{longtable}}

\makeatother

\newcommand{\fa}{\,\forall\,}
\newcommand{\exi}{\,\exists\,}
\newcommand{\Ra}{\Rightarrow}
\newcommand{\xra}{\xrightarrow}
\newcommand{\suml}[2]{\sum\limits_{{#1}}^{{#2}}}
\newcommand{\sums}[1]{\sum\limits_{{#1}}}
\newcommand{\cups}[1]{\bigcup\limits_{{#1}}}
\newcommand*{\bw}[1]{#1\nobreak\discretionary{}{\hbox{$\mathsurround=0pt #1$}}{}}
\newcommand{\sco}{,\ldots,}
\newcommand{\sd}{\bw\cdot\ldots\bw\cdot}
\newcommand{\ha}[1]{\left\langle#1\right\rangle}
\newcommand{\br}[1]{\bigl(#1\bigr)}
\newcommand{\Br}[1]{\Bigl(#1\Bigr)}
\newcommand{\ter}[1]{\textup{(}#1\textup{)}}
\newcommand{\bgm}[1]{\bigl|#1\bigr|}
\newcommand{\Bm}[1]{\Bigl|#1\Bigr|}
\newcommand{\hn}[1]{\left\|#1\right\|}
\newcommand{\bc}[1]{\bigl\{#1\bigr\}}
\newcommand{\ph}{\varphi}
\newcommand{\ep}{\varepsilon}
\DeclareMathOperator{\Hom}{Hom}
\DeclareMathOperator{\End}{End}
\DeclareMathOperator{\Rea}{Re}
\DeclareMathOperator{\const}{const}
\DeclareMathOperator{\supp}{supp}
\DeclareMathOperator{\conv}{conv}
\DeclareMathOperator{\Spec}{Spec}

\begin{document}

\author{Styrt O.\,G.\thanks{Russia, MIPT, oleg\_styrt@mail.ru}}
\title{Convergence in distribution\\
of the product of random variables\\
from an independent sample\\
on a~compact algebraic group}
\date{}
\newcommand{\udk}{512.743.7+512.813.3+512.815.1+512.815.7+519.213+519.214.7}
\newcommand{\msc}{14L30+20G20+20P05+22C05+22E47+28C10+60A10+60B15}

\maketitle

An equivalent condition for the product of elements of an independent random sample on a~compact algebraic group converging in distribution to some
random variable as the sample size increases is obtained. Namely, a~limit distribution exists and is uniform on the support of the parent distribution if
a~random variable with such a~distribution does not belong with the unit probability to any non-trivial coset over an algebraic subgroup that lies in its
normalizer\~ otherwise, it does not exist.

\smallskip


\section*{The notations used}

\lonu{|>{$}l<{$}|l|}
{\ncu{>{\nwl}p{2cm}<{\nwl}}{Expression} \nc{>{\nwl}p{9cm}<{\nwl}}{Meaning}}{%
V^* & the dual space to a~space~$V$
\news
E & the identity operator
\news
\Spec A & the set of eigenvalues of an operator~$A$
\news
J_m(\lambda) & the Jordan cell of size~$m$ with eigenvalue~$\lambda$
\news
X^G & the subset of fixed points of an action $G\colon X$
\news
G_x & the stabilizer of a~point~$x$ for an action of a~group~$G$
\news
I_Y & the indicator function of a~subset~$Y$
\news
\xi\sim D & a~random variable~$\xi$ has distribution law~$D$
\news
\xi\sim\eta & random variables $\xi$ and~$\eta$ are identically distributed
\news
M\xi & the mean of a~random variable~$\xi$
\news
\xra{F} & convergence of random variables in distribution
\news
\text{a.\,s.} & <<almost surely>> (with probability~$1$)
\news
\text{i.\,i.\,d.} & <<independent identically distributed>>
\news
\text{u.\,o.\,s.} & <<unless otherwise stated>>
\news
\mtcc{In a~metric space $(X,\rho)$\:}
\news
\rho(x,Y) & $\inf\bc{\rho(x,y)\colon y\in Y}$\ti the distance from a~point~$x$ to a~subset~$Y$
\news
U_{\ep}(x) & $\bc{y\colon\rho(x,y)<\ep}$\ti the (open) $\ep$\dh neighborhood of a~point~$x$
\news
\ep\text{\dh net} & a~subset~$Y$ such that $\cups{x\in Y}\br{U_{\ep}(x)}=X$}

\section{Introduction}\label{introd}

The papers \cite{qua,fie,rin} are devoted to researching various properties of random variables with values in a~finite set~$G$ on which, (mainly binary)
operations are given. Without loss of generality, we can suppose that $G=\{1\sco k\}$, $k\in\mathbb{N}$. In the space~$\mathbb{R}^k$, take the standard
basis~$Q$ and the affine $(k-1)$\dh dimensional simplex $S^{(k)}:=\conv(Q)$. Then the distribution law of each random variable~$\xi$ on~$G$ is uniquely
defined by the vector $p=p(\xi)\in S^{(k)}$ with coordinates $p_i:=P(\xi=i)$ ($i\in G$). A~distribution is called \textit{uniform} if all coordinates of
its vector are equal (i.\,e. are equal to~$\frac{1}{k}$).

If a~set~$\mathcal{F}$ of operations of arbitrary arities is given on~$G$, then, one can define a~notion of a~\textit{formula} over the set~$\mathcal{F}$
depending on finitely many formal variables. The strict definition of a~formula is inductive\~ it is well known in the particular case $k=2$ for Boolean
functions and literally retranslated onto the general one. In some way, a~formula over~$\mathcal{F}$ represents a~finite expression that involves formal
variables and applying operations from~$\mathcal{F}$ taking the number of arguments into account. The \textit{complexity} of a~formula is defined as the
number of applying operations in it, and the \textit{depth}\ti as the depth of its circuit of functional elements (i.\,e. the possibly greatest distance
between a~root and a~leaf). Finally, a~formula is called \textit{read-once} if each formal variable appears in it at most once.

If a~formula involves $n$ formal variables $x_1\sco x_n$, then it naturally induces an $n$\dh ary operation on~$G$, i.\,e. a~mapping $f\colon G^n\to G$.
Thus, for arbitrary random variables $\xi_1\sco\xi_n$ on~$G$, the random variable $\eta:=f(\xi_1\sco\xi_n)$ is defined on~$G$. In doing so, in the case
of (mutual) independence of the random variables $\xi_1\sco\xi_n$, the distribution vector $p(\eta)\in\mathbb{R}^k$ is uniquely expressed through the
vectors $p(\xi_i)$ ($i=1\sco n$) via an $n$\dh ary operation~$\widehat{f}$ on the simplex~$S^{(k)}$. With one more assumption that the formula is
read-once, the replacement in it of all operations from~$\mathcal{F}$ with the corresponding operations on~$S^{(k)}$ and of the variables~$x_i$\ti with
the vectors $p_i\in S^{(k)}$ gives exactly the operation~$\widehat{f}$.

For read-once formulas of i.\,i.\,d. random variables, the question of convergence of distribution to the uniform one as the complexity and (or) the
depth increases is actively studied. Present one important result in this direction.

\begin{df} A~set with a~binary operation on it is called a~\textit{quasigroup} if left and right multiplications by all elements in it are bijective.
\end{df}

{\newcommand{\refr}{see Theorem~1 in~\cite[\Ss3]{qua}}
\begin{theorem}[\refr]\label{past} Let $*$ be a~quasigroup operation on the set $G=\{1\sco k\}$ and $p\in S^{(k)}$ a~vector with more than $\frac{k}{2}$
non-zero coordinates. Then the distribution of the random variable obtained by substitution of independent random variables with distribution~$p$ to
a~read-once formula over the set~$\{*\}$, as its depth increases, converges exponentially to the uniform one.
\end{theorem}}

Thus, Theorem~\ref{past} states that, unless the above-mentioned convergence holds for i.\,i.\,d. random variables, each of them with
probability~$1$ takes a~value in some subset $B\subset G$ of order at most~$\frac{|G|}{2}$.

In this paper, similar properties of random variables on compact algebraic groups are researched. Consider a~fixed probability space
$(\Omega,\mathcal{A},P)$. A~\textit{random variable} on a~set~$X$ with a~$\sigma$\dh algebra~$\mathcal{B}$ of its subsets is an
$(\mathcal{A},\mathcal{B})$\dh measurable mapping $\xi\colon\Omega\to X$\~ its distribution law is uniquely defined by the \textit{probability measure}
$\mu_{\xi}\colon\mathcal{B}\to[0;1],\,B\to P(\xi\in B)$. Now let $X$ be a~topological space. Then, we will consider random variables
on~$X$ with respect to the corresponding Borel $\sigma$\dh algebra and define \textit{convergence in distribution} $\xi_n\xra[n\to\infty]{F}\xi$ as the
convergence $M\br{f(\xi_n)}\xra[n\to\infty]{}M\br{f(\xi)}$ for any continuous bounded function $f\colon X\to\mathbb{R}$. If, in addition, $X$ is
a~compact set, then the condition of boundedness of the function can be omitted since it follows from continuity.

U.\,o.\,s., assume that, on all finite-dimensional real linear spaces and algebraic groups, the \textit{real} topology is given.

Let $(G,\cdot)$ be a~compact algebraic group. One can define on it the Borel $\sigma$\dh algebra~$\mathcal{B}$ and, also, the \textit{Haar measure}
$\mu\colon\mathcal{B}\to[0;1]$ invariant under left and right multiplications and satisfying the normalization condition $\mu(G)=1$. \textit{Subgroups}
of the group~$G$ will, u.\,o.\,s., be supposed to be \textit{algebraic}. If a~random variable~$\xi$ on~$G$ satisfies the condition $\mu_{\xi}\equiv\mu$,
then its distribution law will be called \textit{uniform} and denoted by $R(G)$.

Each read-once formula of complexity $n-1$ ($n\in\mathbb{N}$) over the set~$\{\cdot\}$ induces the $n$\dh ary operation
$G^n\to G,\,(g_1\sco g_n)\to g_{\tau(1)}\sd g_{\tau(n)}$, $\tau\in S_n$. Substitution to it of independent random variables $\xi_1\sco\xi_n$ with common
distribution law~$D$ gives a~random variable with the distribution law of $\xi_1\sd\xi_n$\~ this law is uniquely defined by $D$ and~$n$ while does not
depend on~$\tau$.

Due to compactness of~$G$, each non-empty family of its subgroups has a~minimal element by inclusion. Let $\xi$ be an arbitrary random variable on~$G$.
The family of all subgroups $H\subset G$ such that $\xi\in H$ a.\,s. is closed under intersection. Hence, there exists the least of such subgroups by
inclusion that will be called the \textit{support} of~$\xi$ (\textit{not.} $\supp\xi$).

The main result of the paper is the following theorem.

\begin{theorem}\label{main} Let $\xi$ be a~random variable on~$G$ with support~$\overline{G}$ and $\xi_n\sim\xi$ \ter{$n\in\mathbb{N}$} independent
random variables. Then the following conditions are equivalent\:
\begin{enumerate}
\item\label{asy} $\xi_1\sd\xi_n\xra[n\to\infty]{F}\eta$ for some random variable~$\eta$ on~$G$\~
\item\label{asyr} $\xi_1\sd\xi_n\xra[n\to\infty]{F}\eta$ where $\eta\sim R(\overline{G})$\~
\item\label{cosu} for each subgroup $H\subset\overline{G}$ and coset $Z\in\br{N(H)}/H$ such that $\xi\in Z$ a.\,s., we have $H=\overline{G}$\~
\item\label{cos} for each subgroup $H\subset G$ and coset $Z\in\br{N(H)}/H$ such that $\xi\in Z$ a.\,s., we have $Z=H$.
\end{enumerate}
\end{theorem}

The implication $\text{\ref{asyr}}\Ra\text{\ref{asy}}$ is obvious. Besides, if \ref{cos} holds, then \ref{cosu} also holds\: for each subgroup
$H\subset\overline{G}$ and coset $Z\in\br{N(H)}/H$ such that $\xi\in Z$ a.\,s., we have $Z=H$, and, hence, $\xi\in H$ a.\,s. that implies
$H\supset(\supp\xi)=\overline{G}\supset H$, $H=\overline{G}$. It remains to prove the implications $\text{\ref{asy}}\Ra\text{\ref{cos}}$
and $\text{\ref{cosu}}\Ra\text{\ref{asyr}}$ that is done further in the paper.

\section{Auxiliary facts}\label{facts}

This section contains a~number of auxiliary statements.

\begin{stm}\label{jor} If $m\in\mathbb{N}$, $\lambda\in\mathbb{C}$, and $|\lambda|<1$, then $\br{J_m(\lambda)}^n\xra[n\to\infty]{}0$.
\end{stm}

\begin{proof} As $n\to\infty$, we have
\begin{equation*}
\fa k\geqslant0\quad\quad\quad C_n^k\lambda^{n-k}\sim\frac{n^k\lambda^{n-k}}{k!}\to0
\end{equation*}
that gives $\br{J_m(\lambda)}^n=\sums{0\leqslant k<m}\Br{C_n^k\lambda^{n-k}\br{J_m(0)}^k}\to0$.
\end{proof}

Let $G$ be a~topological space and $\mathbb{F}$ a~field $\mathbb{R}$ or~$\mathbb{C}$. Denote by $C_{\mathbb{F}}(G)$ the Banach space (over~$\mathbb{F}$)
of all continuous bounded functions $G\to\mathbb{F}$ with the uniform norm. Fix random variables $\eta$ and~$\eta_n$ ($n\in\mathbb{N}$) on~$G$.

For an arbitrary finite-dimensional space~$W$ over~$\mathbb{F}$, all continuous bounded functions $f\colon G\to W$ such that
$M\br{f(\eta_n)}\xra[n\to\infty]{}M\br{f(\eta)}$ form a~space (again over~$\mathbb{F}$) that will be denoted by $L_W(G)$. In particular,
$L_{\mathbb{F}}(G)\subset C_{\mathbb{F}}(G)$.

\begin{prop}\label{oper} For any finite-dimensional spaces~$W_{1,2}$ over~$\mathbb{F}$, operator $A\colon W_1\to W_2$ and function $f\in L_{W_1}(G)$, we
have $(A\circ f)\in L_{W_2}(G)$.
\end{prop}

\begin{proof} Preserving continuity and boundedness is obvious. Besides, if $\xi$ is an arbitrary random variable on~$G$, then
$M\br{(A\circ f)(\xi)}=A\Br{M\br{f(\xi)}}$.
\end{proof}

\begin{imp}\label{lif} Any finite-dimensional space~$W$ over~$\mathbb{F}$ and function $f\colon G\to W$ satisfy the relation
$\br{f\in L_W(G)}\Leftrightarrow\br{\fa\alpha\in W^*\ (\alpha\circ f)\in L_{\mathbb{F}}(G)}$.
\end{imp}

\begin{proof} The operator $E\in\Hom_{\mathbb{F}}(W,W)=W\otimes_{\mathbb{F}}W^*$ has form $\suml{i=1}{m}(A_i\circ\alpha_i)$
($\alpha_i\in W^*$, $A_i\in\Hom_{\mathbb{F}}(\mathbb{F},W)$), and, thus, $f=\suml{i=1}{m}\br{A_i\circ(\alpha_i\circ f)}$. It remains to use
Proposition~\ref{oper}.
\end{proof}

\begin{stm} For any function $f\in C_{\mathbb{F}}(G)$ and random variable~$\xi$ on~$G$, there holds $\Bm{M\br{f(\xi)}}\leqslant\hn{f}$.
\end{stm}

\begin{proof} If $c:=\hn{f}$, then $|f|\leqslant c$, $\bgm{f(\xi)}\leqslant c$, $\Bm{M\br{f(\xi)}}\leqslant M\bgm{f(\xi)}\leqslant M(c)=c$.
\end{proof}

\begin{lemma}\label{clo} The subspace $L_{\mathbb{F}}(G)\subset C_{\mathbb{F}}(G)$ is closed.
\end{lemma}

\begin{proof} Suppose that $f_m\in L_{\mathbb{F}}(G)$ ($m\in\mathbb{N}$), $f_{\infty}\in C_{\mathbb{F}}(G)$ and, also, $f_m\rightrightarrows f_{\infty}$
on~$G$, i.\,e. $c_m:=\hn{f_m-f_{\infty}}\xra[m\to\infty]{}0$. We need to prove that $f_{\infty}\in L_{\mathbb{F}}(G)$.

Set $a_{m,n}:=M\br{f_m(\eta_n)}$ and $a_m:=M\br{f_m(\eta)}$ ($m\in\mathbb{N}\sqcup\{\infty\}$, $n\in\mathbb{N}$). For any random variable~$\xi$ on~$G$
and number $m\in\mathbb{N}$, we have
\begin{equation*}
\Bm{M\br{f_m(\xi)}-M\br{f_{\infty}(\xi)}}=\Bm{M\br{(f_m-f_{\infty})(\xi)}}\leqslant\hn{f_m-f_{\infty}}=c_m.
\end{equation*}
In particular, $|a_{m,n}-a_{\infty,n}|\leqslant c_m$ and $|a_m-a_{\infty}|\leqslant c_m$ ($m,n\in\mathbb{N}$). Recall that $c_m\xra[m\to\infty]{}0$.
Hence, $a_m\xra[m\to\infty]{}a_{\infty}$ and $a_{m,n}\xra[m\to\infty]{}a_{\infty,n}$ \textit{uniformly by} $n\in\mathbb{N}$. Meanwhile,
$a_{m,n}\xra[n\to\infty]{}a_m$ ($m\in\mathbb{N}$). Therefore, $a_{\infty,n}\xra[n\to\infty]{}a_{\infty}$, i.\,e. $f_{\infty}\in L_{\mathbb{F}}(G)$.
\end{proof}

It is obvious that the subspaces $C_{\mathbb{C}}(G)$ and $L_{\mathbb{C}}(G)$ are complexifications of the subspaces $C_{\mathbb{R}}(G)$ and
$L_{\mathbb{R}}(G)$ respectively. It follows that
\begin{equation}\label{reco}
(\eta_n\xra[n\to\infty]{F}\eta)\quad\Leftrightarrow\quad\br{L_{\mathbb{R}}(G)=C_{\mathbb{R}}(G)}\quad\Leftrightarrow\quad\br{L_{\mathbb{C}}(G)=C_{\mathbb{C}}(G)}.
\end{equation}

From now, we will additionally assume that the topological space~$G$ is in fact a~compact algebraic group\~ for its Borel $\sigma$\dh algebra and
Haar measure, save the earlier notations $\mathcal{B}$ and~$\mu$ respectively.

\begin{df} \textit{Matrix coefficients} of the group~$G$ are any functions of type $(\alpha\circ\ph)\colon G\to\mathbb{C}$, where
$\ph\colon G\to\mathbf{GL}(V)$ is an irreducible complex representation and $\alpha$ is a~linear function $\End(V)\to\mathbb{C}$.
\end{df}

The following theorem is known as <<part~I>> of the Peter--Weyl theorem (see, for instance, \cite[\Ss I.5]{Kn}, Theorem~1.12 and Remark after its
statement).

\begin{theorem}\label{PW} The linear span of matrix coefficients of the group~$G$ is dense in $C_{\mathbb{C}}(G)$.
\end{theorem}

\begin{theorem}\label{eqv} The following conditions are equivalent\:
\begin{enumerate}
\item\label{conv} $\eta_n\xra[n\to\infty]{F}\eta$\~
\item\label{matr} for arbitrary irreducible complex representation $\ph\colon G\to\mathbf{GL}(V)$ and linear function $\alpha\colon\End(V)\to\mathbb{C}$,
we have $(\alpha\circ\ph)\in L_{\mathbb{C}}(G)$\~
\item\label{repr} for any irreducible complex representation $\ph\colon G\to\mathbf{GL}(V)$, there holds
\begin{equation}\label{conm}
M\br{\ph(\eta_n)}\xra[n\to\infty]{}M\br{\ph(\eta)}.
\end{equation}
\end{enumerate}
\end{theorem}

\begin{proof} The equivalence $\text{\ref{conv}}\Leftrightarrow\text{\ref{matr}}$ follows from Equation~\eqref{reco}, Lemma~\ref{clo} and Theorem~\ref{PW},
and the equivalence $\text{\ref{matr}}\Leftrightarrow\text{\ref{repr}}$\ti from Corollary~\ref{lif}.
\end{proof}

Let $\ph\colon G\to\mathbf{GL}(V)$ be a~complex representation.

For any independent random variables $\xi_1\sco\xi_n$ on~$G$, the random variables $\ph(\xi_i)$ ($i\bw=1\sco n$) on the space $\End(V)$ are also independent
that implies
\begin{equation*}
M\br{\ph(\xi_1\sd\xi_n)}=M\br{\ph(\xi_1)\sd\ph(\xi_n)}=\Br{M\br{\ph(\xi_1)}}\sd\Br{M\br{\ph(\xi_n)}}.
\end{equation*}
Consequently, for arbitrary element $g\in G$ and random variable~$\eta$ on~$G$, we have
$M\br{\ph(g\eta)}\bw=\Br{M\br{\ph(g)}}\cdot\Br{M\br{\ph(\eta)}}\bw=\br{\ph(g)}\cdot\Br{M\br{\ph(\eta)}}$.

\section{Proofs of the results}\label{prove}

In this section, Theorem~\ref{main} is proved.

Recall that $\text{\ref{asyr}}\Ra\text{\ref{asy}}$ and $\text{\ref{cos}}\Ra\text{\ref{cosu}}$ (see \S\,\ref{introd}). It remains to prove that
$\text{\ref{asy}}\Ra\text{\ref{cos}}$ and $\text{\ref{cosu}}\Ra\text{\ref{asyr}}$.

Take arbitrary random variables $\xi$ and $\xi_n$ ($n\in\mathbb{N}$) from the statement of Theorem~\ref{main}. Set $\eta_n:=\xi_1\sd\xi_n$
($n\in\mathbb{N}$).

\ul{$\text{\ref{asy}}\Ra\text{\ref{cos}}$}.

Take any random variable~$\eta$ on~$G$, subgroup $H\subset G$ and element $b\in N(H)$ such that $\eta_n\xra[n\to\infty]{F}\eta$
and $\xi\in bH$ a.\,s. It is required to prove that $b\in H$.

In the compact algebraic group~$G$, there exists metrics~$\rho$ invariant under left and right multiplications and inducing the standard real topology.

Choose an arbitrary number $\ep>0$.

Since $G$ is a~compact group, it has a~finite $\ep$\dh net~$Y$, and, hence, for some neighborhood~$U$ of type $U_{\ep}(a)$ ($a\in Y$), there holds
$M\br{I_U(\eta)}=P(\eta\in U)>0$.

The function $f\colon G\to\mathbb{R},\,g\to\max\bc{2\ep-\rho(a,g);0}$ is continuous and satisfies $0\leqslant f\leqslant2\ep$ (so, is bounded). Thus,
\begin{align}
\fa g\in U\quad\quad\quad\quad&\rho(a,g)<\ep,\quad f(g)>\ep;\notag\\
f\geqslant\ep I_U;\quad\quad\quad\quad&f(\eta)\geqslant\ep\cdot I_U(\eta);\notag\\
M\br{f(\eta_n)}\xra[n\to\infty]{}{}&M\br{f(\eta)}\geqslant\ep\cdot M\br{I_U(\eta)}>0;\notag\\
\exi n_0\in\mathbb{N}\quad\fa n\geqslant n_0\quad\quad\quad\quad&M\br{f(\eta_n)}>0\label{pos}.
\end{align}

Take an arbitrary positive integer $n\geqslant n_0$. Note that $\eta_n\in b^nH$ a.\,s. If $f|_{(b^nH)}\equiv0$, then $f(\eta_n)=0$ a.\,s. that
contradicts~\eqref{pos}. Hence, for some elements $h_n\in H$ and $g_n:=b^nh_n$, we have $f(g_n)\ne0$, i.\,e. $\rho(a,g_n)<2\ep$.

If $n:=n_0$, then $n,n+1\geqslant n_0$, and, therefore, $\rho(a,g_n),\rho(a,g_{n+1})<2\ep$, $\rho(g_{n+1},g_n)<4\ep$,
$\rho(b,H)\leqslant\rho(b,h_nh_{n+1}^{-1})=\rho\br{b^nbh_{n+1},b^n(h_nh_{n+1}^{-1})h_{n+1}}=\rho(g_{n+1},g_n)<4\ep$.

Since the number~$\ep$ is chosen arbitrarily and the subgroup $H\subset G$ is closed, then $b\in H$.

So, the implication $\text{\ref{asy}}\Ra\text{\ref{cos}}$ is completely proved.

\ul{$\text{\ref{cosu}}\Ra\text{\ref{asyr}}$}.

The subgroup $\overline{G}=\supp\xi\subset G$ with probability~$1$ contains all random variables $\xi$, $\xi_n$ and $\eta_n$ ($n\in\mathbb{N}$).
Therefore, while proving the implication $\text{\ref{cosu}}\Ra\text{\ref{asyr}}$, we can, without loss of generality, assume that $\overline{G}=G$.

Suppose that, for each subgroup $H\subset G$ and coset $Z\in\br{N(H)}/H$ such that $\xi\in Z$ a.\,s., there holds $H=G$. We need to prove that
$\eta_n\xra[n\to\infty]{F}\eta$ where $\eta\sim R(G)$.

The measure~$\mu_{\eta}$ on~$G$ coincides with~$\mu$ that implies
\begin{align*}
&\fa g\in G\quad\fa B\in\mathcal{B}&\quad\quad\quad&\mu_{g\eta}(B)=P(g\eta\in B)=P(\eta\in g^{-1}B)=\mu(g^{-1}B)=\mu(B);\\
&\fa g\in G\quad&&\mu_{g\eta}\equiv\mu\equiv\mu_{\eta},\quad\quad\eta\sim g\eta.
\end{align*}

Let $\ph\colon G\to\mathbf{GL}(V)$ be any irreducible complex representation. Prove the relation~\eqref{conm}.

The subspace $V^G\subset V$, being $G$\dh invariant, equals $0$ or~$V$. If $V^G=V$, then $\ph(G)=\{E\}$, $\ph(\eta_n)=\ph(\eta)=E=\const$, that obviously
implies~\eqref{conm}. So, assume from now that $V^G=0$.

For each $g\in G$, we have $\eta\sim g\eta$, $M\br{\ph(\eta)}=M\br{\ph(g\eta)}=\br{\ph(g)}\cdot\Br{M\br{\ph(\eta)}}$. Therefore,
$\Br{M\br{\ph(\eta)}}V\subset V^G=0$, $M\br{\ph(\eta)}=0$. Set $T:=M\br{\ph(\xi)}$. Then, $M\br{\ph(\eta_n)}=T^n$ ($n\in\mathbb{N}$). The
space~$V$ has a~$G$\dh invariant positively definite Hermitian form $(\cdot,\cdot)$\~ denote by~$\hn{\cdot}$ the corresponding norm.

Suppose that the operator~$T$ has an eigenvalue~$\lambda$ with $|\lambda|\geqslant1$. In this case, there exists a~vector $v\in V\setminus\{0\}$ such
that $Tv=\lambda v$. Note that
\begin{equation*}
\hn{\xi v-\lambda v}^2=\hn{\xi v}^2+\hn{\lambda v}^2-2\Rea(\xi v,\lambda v)=\hn{v}^2+\hn{\lambda v}^2-2\Rea(\xi v,\lambda v)
\end{equation*}
and $M(\xi v,\lambda v)=(Tv,\lambda v)=(\lambda v,\lambda v)=\hn{\lambda v}^2$\~ hence,
\begin{equation*}
M\br{\hn{\xi v-\lambda v}^2}=\hn{v}^2+\hn{\lambda v}^2-2\Rea\br{\hn{\lambda v}^2}=\hn{v}^2-\hn{\lambda v}^2=\br{1-|\lambda|^2}\cdot\hn{v}^2\leqslant0.
\end{equation*}
Therefore, $\xi v=\lambda v$ a.\,s. Denote by~$G_{\ha{v}}$ the subgroup of all elements of~$G$ preserving the line $\ha{v}=\mathbb{C}v$. Obviously,
the kernel of the homomorphism $G_{\ha{v}}\to\mathbf{GL}\br{\ha{v}},\,g\to g|_{\ha{v}}$ is~$G_v$. So, $G_v\lhd G_{\ha{v}}$. All elements of~$G$
taking $v$ to $\lambda v$ form a~coset $gG_v$ ($g\in G_{\ha{v}}\subset N(G_v)$), and $\xi\in gG_v$ a.\,s. By assumption, $G_v=G$, $v\in V^G=0$, that
contradicts the choice of~$v$.

Hence, $|\lambda|<1$ for all $\lambda\in\Spec T$. By Statement~\ref{jor}, each cell~$J$ of the Jordan form of the operator~$T$ satisfies
$J^n\xra[n\to\infty]{}0$. Therefore, $T^n\xra[n\to\infty]{}0$, i.\,e. $M\br{\ph(\eta_n)}\xra[n\to\infty]{}M\br{\ph(\eta)}$.

So, we obtained that any irreducible complex representation $\ph\colon G\to\mathbf{GL}(V)$ satisfies~\eqref{conm}. By Theorem~\ref{eqv},
$\eta_n\xra[n\to\infty]{F}\eta$.

It completely proves the implication $\text{\ref{cosu}}\Ra\text{\ref{asyr}}$ and, thus, the whole Theorem~\ref{main}.

\section*{Acknowledgements}

The author is grateful to Prof. \fbox{E.\,B.\?Vinberg} for exciting interest to algebra.

The author dedicates the article to E.\,N.\?Troshina.


\begin{thebibliography}{9}

\bibitem{qua}
Yashunskii\?A.\,D.
On transformations of probability distributions by read-once quasigroup formulae // Discr. Math. Appl. 2013. Vol. 23. \No\,2. Pp.~211--223.

\bibitem{fie}
Yashunskii\?A.\,D.
On read-once transformations of random variables over finite fields // Discr. Math. Appl. 2015. Vol. 25. \No\,5. Pp.~311--321.

\bibitem{rin}
Yashunskii\?A.\,D.
Convex algebras of probability distributions induced by finite associative rings // Discr. Math. Appl. 2021. Vol. 31. \No\,3. Pp.~223--230.

\bibitem{Kn}
Knapp\?A.\,W.
Representation theory of semisimple groups. Princeton Univ. Press, 1986, 773~p.\\
ISBN\: 0-691-09089-0.

\end{thebibliography}
\end{document}